\documentclass[10pt,a4paper,oneside,english]{amsart}
\usepackage[T1]{fontenc}
\usepackage[latin9]{inputenc}
\usepackage{units}
\usepackage{amsthm}
\usepackage{amssymb}
\usepackage{enumerate}

\makeatletter

\numberwithin{equation}{section}
\numberwithin{figure}{section}
\theoremstyle{plain}

  \theoremstyle{plain}
  
  \theoremstyle{plain}
  
  \theoremstyle{remark}
  
  \theoremstyle{plain}

\@ifundefined{date}{}{\date{}}
\usepackage{amsthm}

\usepackage{cite}%\usepackage{showkeys}

\newtheorem{theorem}{Theorem}[section]

\newtheorem{corollary}[theorem]{Corollary}

\newtheorem{definition}[theorem]{Definition}
\newtheorem{example}[theorem]{Example}

\newtheorem{proposition}[theorem]{Proposition}
\newtheorem{remark}[theorem]{Remark}

\def\x{\overline{x}}

\def\int{\textrm{int}}

\def\x{\overline{x}}

\def\z{\overline{z}}

\def\R{\mathbb{R}}

\def\N{\hbox{{\rm I}\kern-.2em\hbox{{\rm N}}}}

\providecommand{\lemmaname}{Lemma}
  \providecommand{\propositionname}{Proposition}
  \providecommand{\remarkname}{Remark}
\providecommand{\theoremname}{Theorem}

\usepackage{babel}
\providecommand{\lemmaname}{Lemma}
  \providecommand{\propositionname}{Proposition}
  \providecommand{\remarkname}{Remark}
\providecommand{\theoremname}{Theorem}

\usepackage{babel}
\providecommand{\corollaryname}{Corollary}
  \providecommand{\lemmaname}{Lemma}
  \providecommand{\propositionname}{Proposition}
  \providecommand{\remarkname}{Remark}
\providecommand{\theoremname}{Theorem}

\makeatother

\usepackage{babel}
  \providecommand{\corollaryname}{Corollary}
  \providecommand{\lemmaname}{Lemma}
  \providecommand{\propositionname}{Proposition}
  \providecommand{\remarkname}{Remark}
\providecommand{\theoremname}{Theorem}

\begin{document}
	
	\title[Projected Solutions for Quasi Equilibrium Problems]{ On Projected Solutions for Quasi Equilibrium Problems with Non-self Constraint Map}

	\author{M. Bianchi}

	\address{Dipartimento di Matematica per le Scienze Economiche, Finanziarie ed Attuariali, Universit{\`a} Cattolica del Sacro Cuore, Milano, Italy}

	\email{monica.bianchi@unicatt.it}

	\author{E. Miglierina}

	\address{Dipartimento di Matematica per le Scienze Economiche, Finanziarie ed Attuariali, Universit{\`a} Cattolica del Sacro Cuore, Milano, Italy }

	\email{enrico.miglierina@unicatt.it}

	\author{M. Ramazannejad}

	\address{Dipartimento di Matematica per le Scienze Economiche, Finanziarie ed Attuariali, Universit{\`a} Cattolica del Sacro Cuore, Milano, Italy}

	\email{maede.ramazannejad@unicatt.it}
	\begin{abstract}
		\noindent In a normed space setting, this paper studies the conditions under which the projected solutions to a quasi equilibrium problem with non-self constraint map exist. Our approach is based on an iterative algorithm which gives rise to a sequence such that, under the assumption of asymptotic regularity, its limit points are projected solutions.
	Finally, as a particular case, we discuss the existence of projected solutions to a quasi variational inequality problem.	
	\end{abstract}
	

	
	\maketitle

\section{Introduction}
Given a nonempty subset $C$ of a  normed space $X$ and a function  $f:X\times X\to \R $,
the \emph{equilibrium problem}  EP$(f,C)$  is defined as follows: find $\x\in C$ such that
$$
f(\x,y)\ge 0,\quad \forall y\in C.
$$
The (possibly empty) set of solutions of $EP(f,C)$ will be denoted by Sol$(f,C).$
In the sequel $f$ will be called \emph{bifunction}, as usual in literature on equilibrium problems.

The first existence results for equilibrium problems  date back to the seventies and are attributed to Fan \cite{Fan72} and Brezis, Nirenberg and Stampacchia \cite{BreNiSta72} while the term \emph{equilibrium problem} was coined later in the nineties  because
it is equivalent to find the equilibrium points of several problems,
like  minimization problems, saddle point problems, Nash equilibrium problems,
variational inequality problems, fixed point problems, and so forth (see for instance the seminal paper \cite{BluOet94}).

Related to EP$(f,C)$  in literature has been also considered the so-called \emph{quasi equilibrium problem}, that is an equilibrium problem with a constraint set depending on the current point. More precisely, given a set-valued map $\Phi: C \rightrightarrows C,$ the quasi equilibrium problem $QEP(f,\Phi)$ requires to
$$
\text{find}\quad \x \in \Phi(\x)\quad \text{such that}\, f(\x,y) \ge 0, \quad \forall y \in \Phi(\x).
$$
Also this problem  encompasses,  as special cases, many relevant problems which model situations arising in applications like quasi variational inequalities, generalized Nash equilibrium problems, mixed quasi variational-like inequalities (see \cite{AusCotIus2017} and the references therein).

In the literature, to prove existence results for $QEP(f,\Phi)$, the key step goes through fixed-point techniques. On the other hand,  in \cite{Bikapi21} the authors, generalizing an original idea by Konnov in a finite-dimensional setting (see \cite{Kon2019}), proposed  a regularized version
of the penalty method to establish existence results  in reflexive Banach spaces by replacing the quasi equilibrium problem with a sequence of usual equilibrium problems.

However, in the definition of $QEP(f,\Phi)$ the map $\Phi$ is a self-map, i.e.  $\Phi: C \rightrightarrows C,$  while  in many applications this condition is not fulfilled. For this reason, in \cite{AusSulVet2016} the authors  introduced the more general notion of \emph{projected solutions} of a quasi variational inequalities in finite dimension and later in  \cite{CoZu2019}  the  case of quasi equilibrium problem was addressed. Due to possible more realistic applications, recently there has been an increasing interest in studying existence of projected solutions in more general settings (see,  for instance, \cite{Milasi22,BuCo21}).

To fix the idea, by assuming $\Phi:C\rightrightarrows X$, a point $x_0\in C$ is called a \emph{projected solution} of $QEP(f,\Phi)$ if
\begin{eqnarray}\label{QEP}
\text{there exists}\quad z_0\in \mathrm{Sol}(f,\Phi(x_0))\quad\text{such that}\quad x_0\in P_C(z_0)
\end{eqnarray}
where $P_C:X\rightrightarrows C$ is the classical metric projection map.

Setting for semplicity $S(x) = \mathrm{Sol}(f,\Phi(x))$, the projected solutions can be obtained as the fixed points of the map
$\mathcal{T}:C\rightrightarrows C,$ given by the composition of the two maps $P_C$ and $S$:
$$
\mathcal{T}(x)=P_C(S(x)):=\cup_{z\in S(x)}P_C(z).
$$
Note that in case $\Phi:C\rightrightarrows C$, then $\mathcal{T}(x)=S(x)$ for all $x \in C$, and $x_0$ turns out to be   a classical solution of the quasi equilibrium problem.

Having in mind the description above, it is clear that a key tool to study the existence of projected  solutions is fixed points theorems;
in particular, a generalization of the Kakutani fixed point theorem for factorizable set-valued map was recently applied
(see  \cite{Milasi22,AusSulVet2016}),  but the drawback of this approach is the compactness assumption with respect to
the strong topology required to apply the fixed point theorem,  which is very demanding in a normed   space setting.

To overcome this problem, in this paper we propose an iterative procedure to find projected solutions of quasi equilibrium problems in a general normed space setting. The particular case of set-valued variational inequalities will be also investigated.
% and some comparisons with literature discussed.
Our approach is based on an iterative procedure which gives rise to a sequence $\{x_n\} \subseteq C$. In case of asymptotic regularity of this sequence, we can show that its limit points are projected solutions. Moreover, we find a sufficient condition to ensure the asymptotic regularity of the sequence $\{x_n\}$. This condition also implies the strong convergence of the whole sequence $\{x_n\}$, whenever the space $X$ is a Banach space.

The organization of this paper is as follows. In Section 2, we recall the necessary definitions, the concept of metric projection and their related properties. Some existing results for equilibrium problems are also mentioned in this section. The existence of projected solutions to quasi equilibrium problem with non-self constraint map is addressed  in Section 3 where we also discuss some sufficient conditions for asymptotic regularity of the sequence generated by the iterative procedure. Finally, in Section 4, we use the results of the previous section to study the existence of projected solutions to quasi variational inequalities problems in normed spaces.

\section{Preliminaries}

In this section we recall some notions  and results
%on set-valued maps, properties of the metric projection and well known existence results for equilibrium problem
useful for the forthcoming discussions.
In the sequel  $X$ will be a normed space and  $X^*$ its topological dual.  We will denote by  $\to ^s$, $\to^w$  and $\to^{w^*} $ the strong, the weak  and the weak$^*$  convergence, respectively. $B_X$ denotes the closed unit ball of $X$.

\subsection{Some properties of set valued maps}
In the sequel we will deal with the constraint set-valued map $\Phi:C\rightrightarrows X$, $C \subseteq X$,  and,
in case of variational inequalities, with a set valued map  $T: X \rightrightarrows X^*$.
We will  denote by $\mathrm{gph}(\cdot)$ and by $\mathrm{dom}(\cdot )$,  the graph and  the domain of a set-valued map, respectively.
We begin by recalling some definitions.

\begin{definition}
	A set valued map  $\Phi:C\rightrightarrows X$, $C \subseteq  X $,  is said to be
	
	\begin{enumerate}
		\item  (see \cite{Bikapi21})  \emph{sequentially (weakly) lower semicontinuous} at $x\in C$ if for every ($x_n \to^w x$)  $x_n \to^s x$, $\{x_n\}\subseteq C,$ and for every $y\in \Phi(x),$
		there exists a subsequence $\{x_{n_k}\}$ and $y_{k}\in \Phi(x_{n_k})$ such that $y_k\to^s y$;
		
		\item  \emph{(weakly) closed } if its graph is a (weakly) closed subset of $C \times X$.
	\end{enumerate}
\end{definition}

Since Definition 1. involves subsequences, it slightly weakens the classical definition of  lower semicontinuity given, for instance,  in Definition 1.4.2 in \cite{AubFra}. Moreover, it is worth noting that sequential weak lower semicontinuity implies sequential  lower semicontinuity, whereas the reverse implication does not hold, as shown in the following example. Of course, we point out that the two notions are equivalent whenever $X$ is a finite dimensional normed space.

\begin{example}
	Let $X=\ell_2$. Given the set
	$$
	C=\left\lbrace x=\left(x^k\right)_{k=1}^\infty\in \ell_2:x^k\geq 0\,\, \forall k\in \N,x^1+x^2\geq 1\right\rbrace \cap B_X,
	$$
	let us consider the set valued map $\Phi:C\rightrightarrows X$ defined by
	$$
	\Phi(x)=3\frac{x}{\|x\|}+B_X
	$$
	for every $x \in C$.
	It is a simple matter that $\Phi$ is sequentially  lower semicontinuous for every $x \in C$.
	Now, let us consider the point $\bar{x}=\left(\frac{1}{2},\frac{1}{2},0,0,\dots\right)\in C$ and the sequence $\{x_n\}\subset C$ where
	$$
	x_n=\left(\frac{1}{2},\frac{1}{2},0,\dots,0,\underset{(n+3)\text{-th place}}{\underbrace{\frac{\sqrt{2}}{2}}},0\dots\right).
	$$
	It is easy to see that $x_n \to^w \bar{x}$ and that
	$$
	\inf_{\substack{u \in \Phi(\overline{x})\\
			u^{\prime} \in \Phi(x_n)}}
	\|u-u^{\prime}\| \ge 3\sqrt{2-\sqrt{2}} -2 >0.
	$$
	Therefore, we conclude that $\Phi$ is not sequentially weakly lower semicontinuous at $\bar{x}$.
\end{example}

The following proposition gives two sufficient conditions ensuring that a sequentially lower semicontinuous map is also sequentially weakly lower semicontinuous.

\begin{proposition}\label{Pro:semicontinuity}
	Let $C$ be a subset of a normed space $X$ and let  $\Phi:C\rightrightarrows X$ be a sequentially lower semicontinuos map at every $x \in C$. Then, $\Phi$ is sequentially weakly lower semicontinuous at every $x \in C$ whenever one of the following conditions hold:
	\begin{enumerate}
		\item $C$ is a compact set;
		\item $X$ is uniformly convex  and $\Phi(x)\subseteq \Phi(y)$ for every $x,y \in C$ such that $\|x\|\leq \|y\|$.
	\end{enumerate}
\end{proposition}

\begin{proof}
	Since, under condition (1), the proof is straightforward, we prove only the sequential weak lower semicontinuity of $\Phi$ under assumption (2).
	
	Let $x \in C$ and let $\{x_n\}\subseteq C$ be a sequence such that $x_n\to^w x$. Then, it holds
	$$
	\left\|x\right\|\leq \liminf_{n \to \infty}\left\|x_n\right\|.
	$$
	If $\left\|x\right\|< \liminf_{n \to \infty}\left\|x_n\right\|$, there exists a subsequences $\{x_{n_k}\}$ of $\{x_n\}$ and $k_0\in \N$ such that $\Phi(x)\subseteq \Phi(x_{n_k})$ for every $k>k_0$. Let $y \in \Phi(x)$, by considering $y_{k}=y$ for every $k>k_0$, we have a sequence $\{y_k\}$ such that $y_k \in \Phi(x_{n_k})$ and $y_k \to^s y$. If $\left\|x\right\|= \liminf_{n \to \infty}\left\|x_n\right\|$, then there exists a subsequence $\{x_{n_k}\}$ of $\{x_n\}$ such that $\lim_{k \to \infty} \|x_{n_k} \| = \|x\|$. Since $X$ is uniformly convex, we have that $x_{n_k}\to^s x$. Therefore, by sequential lower semicontinuity of $\Phi$, we conclude the proof. \qed
\end{proof}

\begin{definition}
	Let $T: X \rightrightarrows X^*$ and $C \subseteq \mathrm{dom}(T)$. The set valued map $T$ is said to be
	%$T: C\rightrightarrows X^*$, $C \subseteq \mathrm{dom}(T) \subseteq X $,  is said to be
	\begin{enumerate}
		\item \emph{s-$w^*$}-\emph{closed} on  $C$  if, for any $(x_n,x^*_n)\in \mathrm{gph}(T|_C)$ such that $x_n\to ^s x \in C$ and $x^*_n\to^{w^*}x^*,$ one has that $(x,x^*)\in \mathrm{gph}(T|_C)$;
		\item \emph{ bounded }if it maps bounded subsets of its domain into bounded sets of $X^*$;
		\item  \emph{Brezis pseudomonotone} on  $C $  if,  for every $\{x_n\} \subseteq C$ such that $x_n\to^w x\in C,$
		and for every $x^*_n \in T (x_n)$ with
		$$
		\liminf _{n\to +\infty} \langle x^*_n, x-x_n\rangle \ge 0,
		$$
		one has that, for every $y\in C,$ there exists $x^*(y )\in T (x)$ such that
		$$
		\langle x^*(y), y-x\rangle\ge \limsup_{n\to +\infty} \langle x^*_n, y-x_n \rangle;
		$$
		\item   \emph{of type $S_+$} on  $C$   if,  for every $\{x_n\} \subseteq C$ such that $x_n\to^w x\in C,$
		if there exists $x^*_n \in T (x_n),$ with
		$$
		\liminf_{n\to +\infty} \langle x^*_n, x - x_n\rangle \ge 0,
		$$
		it follows that $x_n\to^s x$ in $C$;
		
		\item \emph{uniformly monotone} on $C$ if there exists $\beta : \R_+ \to \R_+$ strictly increasing with $\beta(0)=0$ and $\lim_{t \to +\infty} \beta(t)= +\infty$, such that
		$$
		\langle y_1^*-y_2^*,x_1-x_2 \rangle \geq  \beta (\left\|x_1-x_2 \right\|) \left\|x_1-x_2 \right\|
		$$
		for every $(x_1,y_1^*),\,(x_2,y_2^*)\in \text{graph}(T|_C)$;
		\item \emph{strongly monotone} on $C$ if there exists $k>0$ such that
		$$
		\langle y_1^*-y_2^*,x_1-x_2 \rangle \geq k \left\|x_1-x_2 \right\|^2
		$$
		for every $(x_1,y_1^*),\,(x_2,y_2^*)\in \text{graph}(T|_C)$.
	\end{enumerate}
\end{definition}

In literature  set valued maps of type $S_+$ are involved in the study of existence results for  variational inequalities (see, for instance,  \cite{Cub97})  and they  will play a crucial role in the sequel. In any locally uniformly convex Banach space the duality map is an example of single valued map of type $S_+$ (see Example 6.8, Ch.3 in \cite{Hu97}); other examples of set valued maps of type $S_+$ are the uniformly monotone  ones. For completeness, we prove it in  the following proposition.

{\begin{proposition}\label{prop:unifmon_vs_S+}
		If $T:X\rightrightarrows X^*$  is uniformly monotone on $C \subseteq \mathrm{dom}(T)$,  then $T$ is of type $S_+$ on $C$.
	\end{proposition}
	\begin{proof}
		Let $x\in C$, $x^* \in T(x)$  and let $\{x_n\}\subseteq C$ be a sequence such that $x_n \to^w x$.
		If there exists $x_n^* \in T(x_n)$ such that
		\begin{equation}\label{eq:S+ prop1}
		\liminf_{n \to \infty} \langle x^*_n, x-x_n \rangle \geq 0,
		%\Rightarrow  \limsup_{n \to \infty} \langle x^*_n, x-x_n \rangle \geq 0}
		\end{equation}
		then, by the uniformly monotonicity of $T$, we have
		$$
		\langle x^*-x_n^*,x-x_n \rangle \geq  \beta(\|x-x_n\|) \left\|x-x_n
		\right\|\geq 0.
		$$
		Therefore, it holds
		\begin{equation}\label{eq:S+ prop2}
		\liminf_{n \to \infty}\left(\langle x^*,x-x_n \rangle -\langle x^*_n,x-x_n \rangle\right) \geq \liminf_{n \to \infty}  \beta(\|x-x_n\|) \left\|x-x_n\right\|\geq 0.
		\end{equation}
		Since $(x-x_n) \to^w 0$, we have $\lim_{n \to \infty}\langle x^*,x-x_n \rangle=0$. Hence, the relation (\ref{eq:S+ prop2}) becomes
		$$
		-\limsup_{n \to \infty} \langle x^*_n,x-x_n
		\rangle\geq \liminf_{n \to \infty}  \beta(\|x-x_n\|) \left\|x-x_n
		\right\|\geq 0.
		$$
		By (\ref{eq:S+ prop1}), we obtain
		$$
		0\geq \liminf_{n \to \infty}  \beta(\|x-x_n\|) \left\|x-x_n
		\right\|\geq 0.
		$$
		The same inequalities hold true for  $\limsup_{n \to \infty}  \beta(\|x-x_n\|) \left\|x-x_n
		\right\|$.
		We conclude that
		$$
		\lim_{n \to \infty}\beta(\|x-x_n\|) \left\|x-x_n
		\right\|=0.
		$$
		By considering the properties of $\beta$, we have $x_n\to ^s x$, and $T$ is of type $S_+$.
	\end{proof}
	
	\begin{remark}
		The previous result holds true for the class of strongly monotone set valued maps, since they are special cases of uniformly monotone maps.
	\end{remark}

	\subsection{Metric projection in Banach spaces}
	
	Given a nonempty subset  $C\subseteq X$,  the \emph{metric projection map}  $P_C:X\rightrightarrows C$
	is defined as
	$$
	P_C(x)=\{z\in C:\,\|x-z\|\le \|x-w\|\text{ for all } w\in C\} = \{z\in C:\,\|x-z\| = d(x,C) \}
	$$
	where $d(x,C)= \inf_{w\in C}\|x-w\| $. It is easy to show directly that $P_C(x)$ is a convex subset of $C$, provided $C$ is a convex set.
	
	A sequence $\{x_n \} \subseteq  C$ is called a \emph{minimizing sequence } for $x \in X\setminus C$ if
	$$
	\|x_n-x \| \to d(x,C).
	$$
	In the sequel  $\tau$ will denote either  the strong $(s)$ or the weak topology $(w)$ in $X$.
	\begin{definition}
		A subset $C\subseteq X$ is said to be
		\begin{itemize}
			\item \emph{proximinal} if $P_C(x) \ne \emptyset$ for all $x \in X$;
			\item a \emph{Chebyshev set} if $P_C(x)$ is a singleton for all $x \in X$;
			\item \emph{ approximatively $\tau$-compact} if for each $x\in X \setminus C$ and each minimizing sequence $\{u_n\} \subseteq C$ for $x$, there exists a subsequence $\{u_{n_k} \}$ such that $u_{n_k}  \to^{\tau}  u$, $u \in C$ (see, for instance, \cite{EfiSte1961});
			\item \emph{ boundedly $\tau$-compact}  if for each bounded  sequence $\{u_n\} \subseteq C$, there exists a subsequence $\{u_{n_k} \}$ such that $u_{n_k}  \to^{\tau} u$, $u \in C$.
		\end{itemize}
	\end{definition}
	It is well known that every  $\tau$-compact set is also boundedly $\tau$-compact and it is evident  that every boundedly $\tau$-compact set
	is in particular approximatively $\tau$-compact, since all minimizing sequences are bounded.
	
	\begin{definition}
		
		The metric projection onto a proximinal set $C$ is called \emph{norm-$\tau$ upper semicontinuous}
		(briefly, norm-$\tau$ u.s.c.) at a point $x  \in X$ provided that for each sequence $\{x_n\}$ such that $x_n \to^{s} x$ and each $\tau$-open set
		$V \supseteq P_C(x)$  we have $V \supseteq P_C(x_n)$ eventually (i.e., for $n$ sufficiently large). $P_C$ is called norm-$\tau$ u.s.c. on $X$ if it
		is norm-$\tau$ u.s.c. at each point of $x \in X.$
	\end{definition}

	If $P_C$ is single-valued (i.e., $C$ is a Chebyshev set), then
	norm-$\tau$ u.s.c. reduces to  continuity of the map $P_C$ from  $X$ with
	its norm topology into $C$ with its $\tau$ topology.
	
	We recall the following results, which provides some regularity properties of  metric projection useful in the sequel:
	
	\begin{theorem}\label{Theor:Deutsch}{\cite[Theorem 2.7]{Deu80}}
		Let $X$ be a normed space  endowed with the $\tau$ topology and  $C \subseteq X$  be an approximatively $\tau$-compact set. Then
		\begin{enumerate}[(1)]
			\item $C$ is proximinal;
			\item $P_C $ is norm-$\tau$ u.s.c.
		\end{enumerate}
		Moreover, if  C is boundedly $\tau$-compact,  then
		\begin{enumerate}[(3)]
			\item $P_C(x)$ is $\tau$-compact for each $x \in X$.
		\end{enumerate}
	\end{theorem}

	Since reflexive Banach spaces are characterized by the weak compactness of their unit balls,  each closed and convex (hence weakly closed) subset  $C$  of a reflexive Banach space $X$ is boundedly $w$-compact and, thanks to the Theorem \ref{Theor:Deutsch},  $C$ is proximinal and $P_C$
	is norm-weakly upper semicontinuous.
	
	In addition, if $X$ is  strictly convex,  every closed and convex subset of $X$  is a Chebyshev set and  $P_C$ is norm-weakly continuous.

	Note that if $X$ is an E-space (i.e. a reflexive Banach space,  strictly convex and with the Kadec-Klee property), then every  closed and convex
	subset of $X$ is also approximatively $s$-compact (see Theorem 10.4.6 in \cite{Luc06}).

	\subsection{Existence results for equilibrium problems}
	
	There is an extensive  literature on existence results for equilibrium problems.  Since the focus of our study is to present an iterative procedure to find out projected solutions, at first,  we are not interested  in the more general existence results for equilibrium problems and our reasoning will be based on  a well-known existence result  that was first stated by Brezis, Nirenberg and Stampacchia:
	
	\begin{theorem}(see \cite{BreNiSta72}, Theorem 1)\label{th:EPexistenceB}
		Let $C$ be a nonempty, closed and convex subset of a Hausdorff topological vector space $E$, and $f: C \times C \to \mathbb{R}$ be a bifunction satisfying the following assumptions:
		\begin{enumerate}[(i)]
			\item $f(x,x) \ge 0$ for all $x \in C;$
			\item for every $x \in C$, the set $\{ y \in C : f(x,y) <0 \}$ is convex;
			\item for every $y\in C$, the function $f(\cdot, y)$ is upper semicontinuous on the intersection of $C$ with any finite dimensional subspace $Z$ of $E;$
			\item whenever $x,y \in C$, $x_\alpha$ is a filter on $C$ converging to $x$ and $f(x_\alpha, (1-t)x+ty) \ge 0$ for all $t\in [0,1]$
			and for all $\alpha$, then $f(x,y) \ge 0;$
			\item there exists a compact subset $K$ of $E$, and $y_0\in K\cap C$ such that $f(x,y_0)<0$ for every $x\in C\setminus K.$
		\end{enumerate}
		Then there exists $\x \in C \cap K$ such that
		$$
		f(\x,y) \ge 0 \quad \text{for all }y \in C.
		$$
	\end{theorem}
	
	In the sequel, we will apply Theorem \ref{th:EPexistenceB}  to the case of a normed space $X$ endowed with the weak topology.
	
	Before stating the existence result, let us first recall some useful properties of bifunctions inspired by the analogous
	definitions for maps (see for instance  \cite{ChaWoYao03}).
	\begin{definition}
		A  bifunction  $f: C \times C \to \R$ is said to be
		\begin{enumerate}
			\item \emph{topologically}, or \emph{Brezis pseudomotone} (B-pseudomonotone, for short) on $C$ if for every $\{x_n \}\subseteq  C$ with $x_n \to^w x \in C$ and  such that $\liminf_{n
				\to \infty} f(x_n,x) \ge 0$ it follows that
			$$
			f(x,y) \ge \limsup_{n \to \infty} f(x_n,y) \quad \forall y \in C ;
			$$
			\item of \emph{ type $S_+$} on $C$ if for every $\{x_n \}\subseteq  C$ with $x_n \to^w x \in C$ and such that
			
			$\liminf_{n
				\to \infty} f(x_n,x) \ge 0$ it follows that
			$
			x_n \to ^s x;
			$
			\item \emph{strongly monotone} on $C$ if there exists $k>0$ such that
			$$
			f(x,y) + f(y,x) \le -k \|x-y\|^2 \quad \forall x,y \in C.
			$$
		\end{enumerate}
	\end{definition}

	\begin{remark}\label{remarkUSCPseumonotone}
		
		\begin{enumerate}[i.]
			\item Note that if $f(x,x)=0$ and $f(x,\cdot)$ is sequentially weakly lower semicontinuos for every $x \in C$ (i.e. if $x_n \to ^w \x$, then $f(x,\x) \le \liminf_{n \to \infty}f(x, x_n)$),   then any strongly monotone bifunction is of Type $S_+$. The proof follows that of Proposition \ref{prop:unifmon_vs_S+}.
			
			\item If $f(\cdot,y)$ is  sequentially upper semicontinuous for every $y \in C$ (i.e. if $x_n \to ^s \x$, then $f(\x,y) \ge \limsup_{n \to \infty}f(x_n,y)$) and of type $S_+$, then $f$  is B-pseudomonotone (see Remark 3 in \cite{Bikapi21}).

			\item If $X$ is a normed space equipped with the weak topology and the closed and convex set $C$ is also weakly compact, assumption (v) in Theorem \ref{th:EPexistenceB} trivially
			holds with  $K=C$. Moreover, by Eberlein-\v{S}mulian theorem (see, e.g., Theorem 2.8.6 in \cite{Megginson})
			condition (iv) can be replaced by the following condition stated in terms of sequences:
			\begin{itemize}
				\item whenever $x,y \in C$, $x_n\in C,$ $x_n \rightarrow ^w x$ and $f(x_n, (1-t)x+ty) \ge 0$ for all $t\in [0,1]$ and for all $n$, then $f(x,y) \ge 0.$
			\end{itemize}
			Finally, we point out that the last condition is satisfied under the assumption of Brezis pseudomonotonicity of the bifunction $f$  on $C$ (see Proposition 2 in \cite{Bikapi21}).

		\end{enumerate}
	\end{remark}
	
	Taking into  account the previous remarks, we can state the following existence result for $EP(f,C)$  which will be applied in next section:
	
	\begin{theorem}\label{th:EPexistence2}
		Let $C$ be a nonempty, weakly compact and  convex subset of a  normed space $X$, and $f: C \times C \to \mathbb{R}$ be a bifunction satisfying the following assumptions:
		\begin{enumerate}[(i)]
			\item $f(x,x) =0$ for all $x \in C;$
			\item for every $x \in C$, the set $\{ y \in C : f(x,y) <0 \}$ is convex;
			\item $f(\cdot, y)$ is sequentially upper semicontinuos for all $y \in C$;
			\item $f$ is of type $S_+$ on $C$.
		\end{enumerate}
		Then there exists $\x \in C $ such that
		$$
		f(\x,y) \ge 0 \quad \text{for all }y \in C.
		$$
	\end{theorem}
	Finally note that in case the space $X$ is reflexive, the theorem above holds if $C$ is nonempty, closed, convex,  and bounded.

	\section{Projected solutions for quasi equilibrium problems via an iterative procedure}
	In this section, we aim to investigate problem \eqref{QEP} by the following algorithmic approach that makes it possible to relax some of the requirements of the theorems in \cite{Milasi22,AusSulVet2016}.
	
	\bigskip
	
	\begin{tabular}{l}
		\hline
		\textbf{Algorithm 1:} Projected Solution Procedure\\
		\hline
		(1) For $i=0$ initialize $x_i\in C$. \\
		(2) Solve $EP(f,\Phi(x_i))$. \\
		(3) Choose a point $z_i \in S(x_i)$. \\
		(4) Set $x_{i+1} \in  P_C(z_i)$.\\
		(5) If $x_{i+1} = x_{i}$ stop. Otherwise increase $i$ by 1 and loop to step 2.\\
		\hline
	\end{tabular}
	
	%We restrict ourselves to normed spaces $X$ in order to impose conditions on the set $C$, the constraint map $\Phi$ , and the bifunction $f$ under our
	%study to guarantee the existence of a projected solution for $QEP(f,\Phi)$.
	\bigskip
	
	We recall that a  sequence $\{x_n\}$ in $X$ is \emph{asymptotically regular} (see, for instance, \cite{BorZhu}) if
	\begin{equation}\label{eq:asymptreg}
	\lim_{n\to +\infty} \|x_n-x_{n+1}\| =0.
	\end{equation}
	It is easy to verify that, if an asymptotically regular sequence admits a weakly convergent subsequence $\{x_{n_k}\}$ to $\overline{x}$, then  also $\{x_{n_k+1}\}$  weakly converges to the same $\overline{x}$.

	\begin{theorem}\label{FirstTheoremBanach}
		Let $C$ be a  nonempty, convex, and weakly compact subset of a normed space $X$ endowed with the weak topology.  Let $\Phi: C \rightrightarrows X$ be such that
		\begin{enumerate}[(i)]
			\item $\Phi(x)$ is nonempty and convex for every $x \in C$;
			\item  $\Phi(C)$ is relatively weakly compact;
			\item $\Phi$ is  sequentially weakly lower semicontinuous for every $x \in C$;
			\item $\Phi$ is weakly closed.
		\end{enumerate}
		Let $f:X \times X \to \R$ be a bifunction such that
		\begin{enumerate}[(a)]
			\item $f(x,x)=0$, for every $x\in \Phi(C)$;
			\item for every $x \in \Phi(C)$, the set $\{ y \in \Phi(C) : f(x,y) <0 \}$ is convex;
			\item $f(\cdot,y)$ is sequentially upper semicontinuous  for all $y \in \Phi(C)$;
			\item $f$ is of type $S_+$ on $\Phi(C)$;
			\item $|f(x,y)-f(x,z)| \le h(x) \|y-z\|$, for all $x,y,z \in \Phi(C)$, where $h: X \to \R_+$ is  bounded on bounded sets.
		\end{enumerate}
		%Then Algorithm 1 weakly converges to a projected solution of $QEP(f,\Phi)$.
		If the sequence $\{x_n\}$ generated by Algorithm 1 is asymptotically regular, then it admits a weak limit point which is a projected solution of $QEP(f,\Phi)$.
		
	\end{theorem}
	\begin{proof}
		First, we  observe that $\Phi$ has weakly closed values, because of (iv).
		%it is a weakly closed set valued map.
		On the other hand, for each $x \in C$, $\Phi(x)\subseteq \overline{\Phi(C)}$. We can therefore conclude that $\Phi(x)$ is weakly compact relying on assumption (ii). It is obvious that for each $x \in C$, $\Phi(x)$ satisfies the conditions on the set $C$ in Theorem \ref{th:EPexistence2}. Then by applying Theorem \ref{th:EPexistence2}, we get the existence of solutions for $EP(f,\Phi(x))$ for every $x \in C$, i.e. $S(x) \neq \emptyset$ for each $x \in C$.
		
		The steps (1) and (4) of Algorithm 1 allow us to observe that the sequence $\{x_n\} \subseteq C$ and since $C$ is weakly compact, without loss of generality, we can assume $x_n \to^{w} \x$. The step (3) of the algorithm implies that $z_n \in S(x_n) \subseteq \Phi(x_n)\subseteq \overline{\Phi(C)}$. Using (ii), $\{z_n\}$ is a sequence in a weakly compact set. Consequently, again without loss of generality, we can assume $z_n \to ^{w} \z$. And from (iv), it follows that $\z \in \Phi(\x)$.
		
		Under (iii),  since $x_n \to^{w} \x$ and $\z \in \Phi(\x)$,  there exists $u_n \in \Phi(x_n)$ such  that $u_n\to^s \z$ (again without loss of generality we don't pass to subsequences). From  $z_n \in S(x_n)$, we have $f(z_n, u_n) \ge 0$ and from (e):
		$$
		f(z_n,\z) + h(z_n) \|u_n-\z\| \ge f(z_n,\z) + f(z_n,u_n)  - f(z_n,\z) = f(z_n, u_n) \ge 0.
		$$
		Therefore, since $\{z_n\}$ is bounded, $\liminf_{n \to \infty} f(z_n,\z) \ge 0$ and by assumption (d), we get $z_n \to^s \z$.
		
		Now we show that $\z \in S(\x)$. Take any  $y \in \Phi(\x)$. Again by (iii), there exists $v_n \in \Phi(x_n)$ such that $v_n\to^s y$. With similar steps as above, we obtain
		$$
		f(z_n,y) + h(z_n) \|v_n-y\| \ge f(z_n,y) + f(z_n,v_n)  - f(z_n,y) = f(z_n, v_n) \ge 0,
		$$
		and thus
		$\limsup_{n \to \infty} f(z_n,y) \ge 0$. By (c), $f(\z,y) \ge 0$, and the assertion follows by the arbitrary choice of $y$ in $\Phi(\x)$.
		
		Finally we  show  that $\x \in P_C(\z)$.  By contradiction, let us suppose that $\x \notin P_C(\z)$. Since $P_C(\z)$ is a convex  and
		weakly compact set (see Theorem \ref{Theor:Deutsch}), we apply the Hahn-Banach Theorem to separate $\x$ and $P_C(\z)$. More precisely, there exist two real numbers $\alpha$ and $\beta$, $\alpha <\beta$ and a non-null linear functional $\x^*\in X^*$ such that
		\begin{equation}
		\langle \x^*,\x \rangle<\alpha <\beta < \langle \x^*,x\rangle,
		\end{equation}
		for every $x \in P_C(\z)$. Now, let
		$$
		\eta = \min_{x\in P_C(\z)}\langle \x^*,x\rangle.
		$$
		The number $\eta$ exists since $P_C(\z)$ is a weakly compact set. It holds $\beta <\eta$, hence we consider the weakly open set
		$$
		V=\left\lbrace v\in X:\langle \x^*,v \rangle>\beta+\frac{\eta-\beta}{2}  = \frac {\beta + \eta} 2 \right\rbrace.
		$$
		It is clear that $V \supseteq P_C(\z)$. Since, by Theorem \ref{Theor:Deutsch}, the set valued map $P_C$ is norm-weakly u.s.c., and there exists a suitable subsequence $\{z_{n_k}\} $ of the original sequence $\{z_n\}$ such  that $z_{n_k} \to ^s \z$, there exists $n_0 \in \N$ such that $V \supseteq P_C(z_{n_k})$ for every $n\geq n_0$.
		Now, pick $x_{n_k+1} \in P_C(z_{n_k})$, hence
		$$
		\langle \x^*, x_{n_k+1}\rangle >\frac{\beta+\eta}{2}> \beta
		$$
		for every $n\geq n_0$. On the other hand, since $x_{n_k} \to^w \x$, and $\{x_n\}$ is asymptotic regular,
		$$
		\lim_{n \to \infty}\langle \x^*, x_{n_k+1}\rangle=\langle \x^*, \x\rangle <\alpha <\beta,
		$$
		which is a contradiction.
		
		To sum up,  given the presence of $\z \in S(\x)$ and $\x \in P_C(\z)$, $\x$ is a projected solution.

	\end{proof}
	
	\begin{remark}
		If $C$  is a Chebyshev set, the metric projection $P_C$ is single valued and norm-weakly continuous and we can avoid the separation argument of the last part of the proof.  Indeed,  from $z_{n_k} \to^s\z$, we get $P_C(z_{n_k}) \to ^{w} P_C(\z)$. But $P_C(z_{n_k})=x_{n_k+1} \to ^{w} \x$ and thus $\x=P_C(\z)$. Being $\z \in S(\x)$, this proves that $\x$ is a projected solution.
	\end{remark}

	\begin{remark}
		The assumption about asymptotic regularity of $\{x_n\}$ in Theorem \ref{FirstTheoremBanach} can be weakened by assuming an asymptotic regularity condition with respect to weak convergence; namely, we may require that
		$$
		\lim_{n \to \infty}\left|\langle x^*,x_n-x_{n+1}\rangle \right| =0
		$$
		for every $x^*\in X^*$.
	\end{remark}

	\begin{remark}
		Taking into account Theorem 16.12 in \cite{CharalAliprantis}, note that  assumptions (ii) and (iv) in the theorem above entail that the set valued map $\Phi$ is norm-weakly upper semicontinuous.
		%Theorem 16.12 in \cite{CharalAliprantis} provides an interesting characterization of $\Phi$. With two assumptions (ii) and (iv) in the theorem above, it is easy to have that the set valued map $\Phi$ is norm$\times$weak upper semicontinuous using this theorem.
	\end{remark}

	We now discuss some particular instances of Theorem \ref{FirstTheoremBanach}.

	Since every bounded and boundedly $w$-compact subset of $X$ is weakly compact,  we get the following.
	
	\begin{corollary}\label{SecondTheoremBanach}
		Let $C$ be a  nonempty, convex, bounded, and boundedly w-compact subset of a normed space $X$. Let $f:X \times X \to \R$ be a bifunction and $\Phi: C \rightrightarrows X$ be a set-valued map with nonempty and convex values while $\overline{\Phi(C)}$ is bounded and boundedly w-compact. Suppose that the conditions (iii)-(iv) and (a)-(e) of Theorem \ref{FirstTheoremBanach} are satisfied.
		%Then Algorithm 1 weakly converges to a projected solution of $QEP(f,\Phi)$.
		If the sequence $\{x_n\}$ generated by Algorithm 1 is asymptotically regular, then it admits a weak limit point which is a projected solution of $QEP(f,\Phi)$.
	\end{corollary}
	%\begin{proof}
	%Since every bounded and boundedly w-compact set is a weakly compact set, $C$ and $\Phi(C)$ are weakly compact sets. Hence, by following the proof of Theorem \ref{FirstTheoremBanach}, Algorithm 1 converges to a projected solution of $QEP(f,\Phi)$.
	%\end{proof}
	
	In a reflexive normed space the assumption about relative weak compactness of $\Phi(C)$ can be replaced by a boundedness assumption.
	
	\begin{theorem}\label{maintheorem}
		Let $C$ be a nonempty, convex, bounded, and closed subset of a reflexive normed space $X$. Let $f:X \times X \to \R$ be a bifunction and $\Phi: C \rightrightarrows X$ be a set-valued map with nonempty, and convex values with  $\Phi(C)$  bounded. Suppose that the conditions (iii)-(iv) and (a)-(e) of Theorem \ref{FirstTheoremBanach} are satisfied.
		%Then Algorithm 1 weakly converges to a projected solution of $QEP(f,\Phi)$.
		If the sequence $\{x_n\}$ generated by Algorithm 1 is asymptotically regular, then it admits a weak limit point which is a projected solution of $QEP(f,\Phi)$.
	\end{theorem}
	\begin{proof}
		Given that every convex, closed, and bounded set in a reflexive space is weakly compact, it is simple to deduce that $\Phi(x)$ is weakly compact for every $x \in C$. Therefore, by using Theorem \ref{th:EPexistence2}, we have $S(x) \neq \emptyset$.
		
		Due to the structure of Algorithm 1 and the boundedness of the set $C$, $\{x_n\}$ has a weakly convergent subsequence since it is a bounded sequence in a reflexive space. Without losing generality, we consider $x_n \to^{w} \x$.	
		
		Furthermore, step 3 of Algorithm 1 and the fact that  $\Phi(C)$ is bounded, indicate that $\{z_n\} \subseteq S(x_n) \subseteq \Phi(x_n) \subseteq \Phi(C) $ is bounded, too and, again without loss of generality, we can set $z_n \to ^{w} \z$.
		
		We skip the remaining parts of the proof since they proceed similarly to the proof of Theorem \ref{FirstTheoremBanach}.
		
	\end{proof}

	When  $X$ is a finite dimensional space, assumption (d) is not necessary to show the existence of projected solutions because the strong and weak convergence coincide. Consequently, we can draw the following result  to  compare with Corollary 3 in \cite{CoZu2019}.
	
	\begin{corollary}\label{cor:finitedim}
		
		Let $C$ be a nonempty, convex, bounded, and closed subset of $\R^n$. Let $\Phi: C \rightrightarrows \R^n$ be a lower semicontinuous and closed map with nonempty and convex values while $\Phi(C)$ is bounded. Suppose that $f:\R^n \times \R^n \to \R$ is a bifunction satisfying  (a)-(c) and (e) of Theorem \ref{FirstTheoremBanach}.
		%\begin{enumerate}
		%\item $f(x,x)=0$, for every $x\in \Phi(C)$;
		%\item for every $x \in \Phi(C)$, the set $\{ y \in \Phi(C) : f(x,y) <0 \}$ is convex;
		%\item $f(\cdot,y)$ is upper semicontinuous, for all $y \in \Phi(C)$;
		%\item $|f(x,y)-f(x,z)| \le h(x) \|y-z\|$, for all $x,y,z \in \Phi(C)$, where $h: \R^n \to \R$ is positive and bounded on bounded sets.
		%\end{enumerate}
		%Then Algorithm 1 converges to a projected solution of $QEP(f,\Phi)$.
		If the sequence $\{x_n\}$ generated by Algorithm 1 is asymptotically regular, then it admits a limit point which is a projected solution of $QEP(f,\Phi)$.
	\end{corollary}
	
	We dedicate the last part of this section discussing the assumption requiring that the sequence generated by Algorithm 1 satisfies condition \eqref{eq:asymptreg}.
	
	First, we note that if the assumption about asymptotic regularity is not fulfilled, then the sequence $\{x_n\}$ generated by Algorithm 1 may have limit points which are not  projected solutions, as the following simple example shows.
	
	\begin{example}
		Let us consider $C=\{x=(x^1,x^2)\in \R^2:-1 \leq x^1 \leq 1, x^2=0\}$, $f(x,y)= y^2-x^2$ (where $x=(x^1,x^2),\, y=(y^1,y^2) \in \R^2$) and
		$$
		\Phi(x)=\Phi((x^1,x^2)) = \{(y^1,y^2)\in \R^2:y^1=-x^1, \, 1 \leq y^2 \leq 2\}.
		$$
		All the assumptions about $C,\Phi$ and $f$ in Corollary \ref{cor:finitedim} are satisfied. Simple considerations show that the solution for the quasi equilibrium problem is unique for every $x\in C$ and the solution map $S$ is defined by $S(x)=(-x^1,1)$. Moreover, we check at once that the only projected solution of our problem is the point $(0,0)$.

		Now, let us consider a point $x_0=(x_0^1,0) \in C $ with $x_0^1\neq 0$.  We have $\Phi(x_0)=\{(y^1,y^2)\in \R^2:y^1=-x_0^1,\, 1 \leq y^2 \leq 2\}$. Hence, we obtain $z_0=S(x_0)=(-x_0^1,1)$ and  so $P_C(z_0)=x_1 = (-x_0^1,0)$. Now $\Phi(x_1)=\{(y^1,y^2)\in \R^2:y^1=x_0^1,\, 1 \leq y^2 \leq 2\}$. Therefore, we get  $z_1=S(x_1)=(x_0^1,1)$ and  so $P_C(z_1)=x_2 = (x_0^1,0)$.
		By iterating the procedure of Algorithm 1 we obtain the sequences
		$$
		x_n= \begin{cases}
		x_0 & \text{ if } n \text{ is even } \\
		-x_0 & \text{ if } n \text{ is odd;}
		\end{cases} \quad \quad
		z_n= \begin{cases}
		(-x_0^1,1) & \text{ if } n \text{ is even } \\
		(x_0^1,1) & \text{ if } n \text{ is odd.}
		\end{cases}
		$$
		It is easy to check that the sequence $\{x_n\}$ is not asymptotically regular and it admits two convergent subsequences: $x_{2n} \to x_0$ and $x_{2n+1} \to -x_0$. Moreover, neither $x_0$ or $-x_0$ are projected solutions.

		Finally, we point out that in this example, Algorithm 1 does not converge to a projected solution for any choice of the starting point $x_0\neq0$.
		
	\end{example}

	The previous example shows that there are problems where, for almost all starting points, Algorithm 1 generates sequences that are not asymptotically regular, even if the problem admits projected solutions.
	Our aim is therefore to find sufficient conditions that give the asymptotic regularity of the sequence $\{x_n\}$  generated by Algorithm 1, whatever we choose a  starting point in $C$.
	
	We recall that a map $T: C \subseteq X \rightrightarrows X$ with $C \subseteq \rm{dom}(T)$,  is said to be \emph{L-Lipschitz} on $C$ if given $x,y \in C$, for each $x' \in T(x)$, there exists $y' \in T(y)$ such that
	$$
	\| x'-y' \|\le L \, \| x-y \|
	$$
	or equivalently,
	$$
	T(x) \subseteq T(y) + L \, \|x-y\|B_X, \quad \forall x,y \in C.
	$$
	%where $B_X$ denotes the closed unit ball in $X$.
	In particular, the map is said to be \emph{non-expansive} if $L=1$ and a \emph{contraction} if $L<1$.
	
	An example of non-expansive map (single valued)} is the metric projection $P_C$ on a closed and convex set in a Hilbert space. On the other hand the metric projection map $P_C$ is no longer non-expansive when we consider a general normed space (see \cite{Alber} and the reference therein).

If we assume further assumptions on the maps $P_C$ and $S$,  we can prove that, for each choice of the starting point $x_0 \in C$, it is possible to make the choices in steps (3) and (4) of Algorithm 1 such that the sequence $\{ x_n\}$ generated is asymptotically regular.

\begin{proposition} \label{Proposition contraction} Let $X$ be a normed space. Let us suppose that the projection $P_C$ is nonexpansive and the solution map $S$ is a contraction.  Then, for every choice of $x_0\in C$ there exists a sequence $\{x_n\}$  generated by Algorithm 1 which is asymptotically regular. Moreover, if $X$ is a Banach space, $\{x_n\}$ strongly converges to a projected solution.
\end{proposition}

\begin{proof}
	Let us pick a point $x_0\in C$ and let $z_0 \in S(x_0)$. According to the nonexpansivity of $P_C$, chose $x_1 \in P_C(z_0)$ such that
	$$
	\|x_0-x_1\| \le  \|x_0-z_0\|.
	$$
	Now, we can choose  $z_1 \in S(x_1)$  such that
	$$
	\|z_0-z_1\| \le L \|x_0-x_1\| \le L \|x_0-z_0\|
	$$
	Again, we choose $x_2 \in P_C(z_1)$ such that
	$$
	\|x_1-x_2\| \le  \|z_0-z_1\| \le L \|x_0-z_0\|
	$$
	and $z_2 \in S(x_2)$ such that
	$$
	\|z_1-z_2\| \le L \|x_1-x_2\| \le L^2 \|x_0-z_0\|.
	$$
	By repeating this procedure, after $n$ steps, we get
	$$
	\|x_n-x_{n+1}\| \le L^n \|x_0-z_0\|
	$$
	and, since $L<1$, the sequence $\{x_n\}$ is asymptotically regular.  By Theorem \ref{FirstTheoremBanach}, the sequence $\{x_n\}$  admits subsequence which  weakly converges to a projected solution $\x$.
	
	In addition, for  $n>m$ we have
	\begin{align*}
	\|x_m-x_n\|  & \le \|x_{m} -x_{m+1} \| + \cdots + \|x_{n-1}-x_n\| \\
	&\le L^m (1+ L + \cdots + L^{n-m-1}) \|x_0-z_0\| \\
	&\le \frac {L^m} {1-L} \|x_0-z_0\|.
	\end{align*}
	
	Therefore $\{x_n\}$ is a Cauchy sequence and, if $X$ is a Banach space, the whole sequence $\{x_n\}$ strongly converges to the  projected solution $\x$.
\end{proof}

\begin{remark}
	
	Under the assumptions of Proposition \ref{Proposition contraction}, the set-valued map $\mathcal{T}:C\rightrightarrows C$, defined by $\mathcal{T}(x)=P_C(S(x))$, is  a contraction. If $\mathcal{T}$ is closed valued (i.e. $\mathcal{T}(x)$ is a closed set for all $x \in C$), Nadler fixed point theorem (\cite{Nad}) proves that $\mathcal{T}$ has a fixed point in $C$. It is easy to see that this point is a projected solution. Algorithm 1  gives a procedure to find it.
	
\end{remark}

Now, we point out a set of conditions such that our algorithm works in a Hilbert space framework. This result is based on sufficient conditions ensuring lipschitzian property for the solution map $S$  that can be obtain arguing as in the proof of Theorem 2.2.1 in \cite{ManRia} with a particular choice of the control data.

\begin{corollary}
	Let $X$ be a Hilbert space and $C\subset X$ be nonempty, convex and weakly compact set. Let us suppose that
	\begin{itemize}
		\item the map $\Phi$ satisfies assumptions (i)-(iv) of Theorem \ref{FirstTheoremBanach} and there exits $L>0$ such that
		\begin{equation}\label{equation 1 corollary hilbert}
		\Phi(x)\subseteq \Phi(y)+L\left\|x-y\right\|B_X
		\end{equation}
		for every $x,y \in C$;
		\item the bifunction $f$ satisfies assumptions (a)-(d) of Theorem \ref{FirstTheoremBanach}. Moreover let us assume that $f$ is strongly monotone on $\Phi(C)$, i.e., there exists $m>0$ such that
		\begin{equation}\label{equation 2 corollary Hilbert}
		f(x,y)+f(y,x)\leq-m\left\| x-y\right\|^2
		\end{equation}
		for every $x,y \in \Phi(C)$,  and the following additional condition holds: there exists $R>0$ such that
		\begin{equation}\label{equation 3 corollary hilbert}
		\left|f(x,y)-f(x,y')\right|\leq R\left\|y-y'\right\|^2
		\end{equation}
		for every $x,y,y'\in \Phi(C)$.
	\end{itemize}
	If $\sqrt{\frac{2RL}{m}}<1$, then, the sequence generated by the Algorithm 1 strongly converges to a projected solution $\x$.	
\end{corollary}

\begin{proof}
	First of all, arguing as in the beginning of the proof of Theorem \ref{FirstTheoremBanach}, we have that $S(x)\neq \emptyset$ for every $x \in C$. Moreover, strong monotonicity assumption implies easily that $S$ is a single valued map. Now, following the proof of  Theorem 2.2.1 (step II and III) in \cite{ManRia} we can  shows that
	$$
	\left\|S(x)-S(y)\right\|\leq \sqrt{\frac{2RL}{m}}\left\|x-y\right\|
	$$
	for every $x,y \in C$. Proposition \ref{Proposition contraction} concludes the proof.
\end{proof}

The next example shows that  Algorithm 1 generates a converging sequence $\{x_n\}$ even if the solution map $S$ is not a contraction. It is worth pointing out that in the following example, the convergence of $\{x_n\}$ does not depend on the choice of the starting point.  The problem considered in the example comes from Example 2.1 in \cite{AusSulVet2016}.

\begin{example} \label{ex:Aussel}Let $C=\{x=(x^1,x^2) \in \R^2: 0\le x^1 \le 1, 0\le x^2\le 1, x^1+x^2 \ge 1 \}$,
	$$
	\Phi(x)= Q + \frac 2 {\|x\|} x,
	$$
	where $Q=\{(x^1,x^2) \in \R^2: 0\le x^1 \le 1, 0\le x^2\le 1 \}$ is the unit square, and $f: \R^2 \times \R^2 \to \R$ given by $f(x,y)= \langle x, y-x \rangle$.

	It is easy to see that the map $\Phi$ and the bifunction $f$ satisfy the assumptions (i)-(iv) and (a)-(e) in Theorem \ref{FirstTheoremBanach}, respectively. Moreover,  standard computations show that, for any $x \in C,$ the solution of $EP(f, \Phi(x))$ is unique and given by
	$$
	S(x)=\frac 2 {\|x\|} x.
	$$
	The solution map $S$ is not a contraction, since
	$$
	\left\|S\left((1,0)\right)- S((0,1))\right\|=\left\|(2,2)\right\|=2\sqrt 2$$
	while   $\|(1,0)-(0,1)\| = \sqrt 2$. The projection on $C$ of the unique solution $S(x)$ is given by
	$$
	P_C(S(x)) =P_C(S((x^1,x^2)))=
	\begin{cases}
	\left (1, \frac{2x^2}{\|x\| }\right )  & \text{ if }  0 \le x^2 < \frac {\sqrt 3} 3 x^1 \\
	\left (1,1 \right ) & \text{ if }   \frac {\sqrt 3} 3 x^1 \le x^2 \le \sqrt{3} x^1 \\
	\left (\frac{2x^1}{\|x\|},1 \right ) & \text{ if }    \sqrt{3} x^1 < x^2 \le 1.
	\end{cases}
	$$
	Starting from a generic point $x_0=(x_0^1,x_0^2) \in C$, in order to analyse the behaviour of the sequence $\{x_n=(x_n^1,x_n^2)\}$ generated by Algorithm 1, we distinguish between several cases:

	\begin{enumerate}	
		\label{1.}
		\item $x_0^2=0$: then, $x_0^1=1$,  $P_C(S((1,0)))= (x_1^1,x_1^2) = (1,0)$. Hence, the sequence $\{(x_n^1,x_n^2)\}$  is constantly equal to $(1,0)$, and $(1,0)$ is a projected solution;
		\item  $0 < x_0^2 < \frac {\sqrt 3} 3 x_0^1$: we have $P_C(S((x_0^1,x_0^2)))= (x_1^1,x_1^2) = \left (1, \frac{2x_0^2}{\|x_0\| } \right )$. Then, if
		$
		\frac {x_0^1}{x_0^2} \le \sqrt {11}
		$
		we have
		$$
		P_C(S((x_1^1,x_1^2)))= (x_2^1,x_2^2) = (1,1)
		$$ and the sequence $\{(x_n^1,x_n^2)\}$  is constant for $n \ge 2$,  and $(1,1)$ is a projected solution. If $\frac {x_0^1}{x_0^2} > \sqrt {11}$, we get
		$$
		P_C(S((x_1^1,x_1^2)))= (x_2^1,x_2^2) = \left (1,\frac{4x_0^2}{\sqrt{\left(x_0^1\right)^2+5\left(x_0^2\right)^2} }\right ).
		$$
		Again, if $\frac {x_0^1}{x_0^2} \le \sqrt {43}$ then $P_C(S((x_2^1,x_2^2)))= (x_3^1,x_3^2) = (1,1)$ and the sequence $\{(x_n^1,x_n^2)\}$  is constant for $n \ge 3$,  and ($1,1)$ is a projected solution.
		In general, the  $k$-step of the Algorithm 1 gives
		$$
		(x_k^1,x_k^2)= \left (1, \frac{2^kx_0^2}{\sqrt{\left(x_0^1\right)^2+\frac{(4^k-1)}{3} \left(x_0^2\right)^2} } \right ).$$
		Since for any $x_0^2>0$ there exists $k_0 \in \N$ such that
		$$
		\frac{2^kx_0^2}{\sqrt{\left(x_0^1\right)^2+\frac{(4^k-1)}{3} \left(x_0^2\right)^2} } \ge \frac{\sqrt{3}} {3},
		$$
		we can conclude that the sequence $\{x_n=(x_n^1,x_n^2)\}$  is constantly equal to $(1,1)$ for $n \ge k_0$,  and ($1,1)$ is a projected solution.
		\item  $\frac {\sqrt 3} 3 x_0^1 \le x_0^2 \le \sqrt{3} x_0^1$: we have  $P_C(S((x_0^1,x_0^2)))= (x_1^1,x_1^2) = (1,1)$ and $P_C(S((1,1))) = (1,1) = (x_2^1,x_2^2)$. Thus the sequence $\{x_n=(x_n^1,x_n^2)\}$  is constant for $n \ge 1$,  and ($1,1)$ is a projected solution;
		\item $\sqrt{3} x_0^1 < x_0^2 \le 1$ with $x_0^1\neq 0$: by following the same reasoning as in case 2., we conclude that  $(x_n^1,x_n^2)=(1,1)$ eventually,  and $(1,1)$ is a projected solution.
		\item $x_0^1=0$: then, $x_0^2=1$ and  as in  case 1. we conclude that $(x_n^1,x_n^2)=(0,1)$ for every $n$,  and $(0,1)$ is a projected solution.
	\end{enumerate}
\end{example}

We conclude this section by giving another example where Algorithm 1 works. This example is quite simple but it has some interest since its setting is infinite dimensional and $C$ is a weakly compact set that is not compact.

\begin{example}
	Let $X= \ell_2$, $C=\{x=(x^k) \in X: x^k\geq 0  \}\cap B_{X}$,  
	$$
	\Phi(x)= \left(3-\left\|x\right\|\right)\tilde{x}+ \left\lbrace u=(u^k) \in X: 0\leq u^k \leq \frac{1+\left\|x\right\|}{ k} \right\rbrace,
	$$
	where $\tilde{x}=\left(\frac{1}{k}\right)_{k=1}^\infty \in X$ and $f:\ell_2 \times \ell_2 \to \R$  given by  $f(x,y)= \langle x, y-x \rangle$.
	Since $\|\tilde{x}\|=\frac{\pi}{\sqrt{6}}$ we get that $\Phi(x)\cap C=\emptyset$ for every $x \in C$.
	Moreover, the following properties hold:
	\begin{itemize}
		\item $C$ is nonempty, convex and weakly compact subset of $X$;
		\item $\Phi(x)$ is a nonempty and convex set for every $x\in C$;
		\item $\Phi(C)=2\tilde{x}+ \left\lbrace u=(u^k) \in X: 0\leq u^k \leq 2 \tilde{x}^k \right\rbrace$ is a compact set;
		\item $\Phi$ is sequentially weakly lower semicontinuous for every $x\in C$. Indeed, this property follows from Proposition \ref{Pro:semicontinuity}, since $\Phi$ is sequentially lower semicontinuous, $X$ is a Hilbert space, hence uniformly convex and $\Phi(x)\subseteq \Phi(y)$ for every $x,y \in C$ such that $\|x\|\leq\|y\|$,
		\item $f$ satisfies the assumptions (a)-(e) in Theorem \ref{FirstTheoremBanach}.
	\end{itemize}

	Easy computations show that, for any $x \in C,$ the solution of $QEP(f, \Phi)$  is unique and given by
	$$
	S(x)=\left(3-\|x\|\right)\tilde{x}
	$$
	for every $x \in C$. 	The solution map $S$ is not a contraction, since
	$$
	\left\|S\left(x\right)- S(y)\right\|\leq\frac{\pi}{\sqrt{6}}\|x-y\|.
	$$
	The projection on $C$ of the unique solution $S(x)$ is given by
	$$
	P_C(S(x)) =\frac{\tilde{x}}{\|\tilde{x}\|}.
	$$
	By taking in account these facts, it follows immediately that the sequence generated by Algorithm 1 becomes constantly equal to $\frac{\tilde{x}}{\|\tilde{x}\|}$ after two steps, whatever we choose the starting point $x_0$.

\end{example}

\section{Set-valued quasi variational inequalities}

In this section we apply our result about the existence of projected solutions of a general quasi equilibrium problem to a special case. Namely, we deal with quasi variational inequalities which have been  extensively studied in recent literature also for their interesting applications (see \cite{Len13,KanSteck19}).

First of all, we recall  a special bifunction. Indeed, any set-valued map  $T:X\rightrightarrows X^*$ properly interacts with the representative bifunction $G_T:X\times X\to \R\cup \{\pm \infty\},$ given by
$$
G_T(x,y)=\sup_{x^*\in T(x)}\langle x^*, y-x\rangle.
$$

The  \emph{set-valued quasi variational inequality} $QVI(T,\Phi)$: find $\x \in \Phi(\x)$  such that
\begin{equation}\label{eq:VIsup}
\sup_{x^* \in T(\x)} \langle x^*, y-\x \rangle \ge 0 \quad \text{for all\;} y \in \Phi(\x),
\end{equation}
can be seen as the quasi equilibrium problem connected to the bifunction $G_T$ and the set valued map $\Phi: C \rightrightarrows X,$ where $\Phi(C) \subseteq \rm{dom}(T)$.

We remind that a projected solution for $QVI(T,\Phi)$ in \eqref{eq:VIsup} is  a point $x_0  \in P_C(z_0)$ where $z_0$ solves the following set-valued variational inequality: find $z \in \Phi(x_0)$ such that
\begin{equation}\label{eq:VIsup1}
\sup_{x^* \in T(z)} \langle x^*, y- z \rangle \ge 0 \quad \text{for all\;} y \in \Phi(x_0).
\end{equation}

To determine the existence of projected solutions for the quasi variational inequality $QVI(T,\Phi)$, this section applies the findings of Section 3 to the bifunction  $G_T$. Here, we investigate the same results as \cite{Bikapi21} under different conditions in the normed space $X$, and our analysis lead us to conclude that the projected solution for $QVI(T,\Phi)$ in the normed space $X$ is also achievable.

%Here, we briefly review some of the findings in [3] that lead us to believe that the projected solution for $QVI(T,\Phi)$ in the real Banach space is also achievable.

%Here, we take a look at a few of the results in \cite{Bikapi21} that inspire us to draw the conclusion that the projected solution for $QVI(T,\Phi)$ in the Banach space exists.

%After reviewing the results in \cite{Bikapi21}, we apply their proof technique to demonstrate the existence of the projected solution for $QVI(T,\Phi)$ in the Banach space.

We first remark that $G_T(x,x)=0$ and that $G_T(x, \cdot)$ is convex for all $x \in X$. Motivated by Proposition 7 and Theorem 5 in \cite{Bikapi21}, we show how $G_T$, based on a few conditions on $T$, fulfills  properties (c), (d), and (e) in Theorem \ref{FirstTheoremBanach}.

%Following, inspired by Proposition 7 in \cite{Bikapi21}, we indicate how $G_T$ can be shown to have the properties of c, d, and e in Theorem \ref{FirstTheoremBanach} based on a few constraints of $T$.

%\begin{proposition}\label{sufps}{\cite[Proposition 7]{Bikapi21}}
%Let $X$ be a reflexive Banach space, $T:X\rightrightarrows X^*,$ and $C\subset \mathrm{dom}(T)$ be a nonempty, closed and convex set. Suppose that $T$ is convex-valued, s-w-closed, bounded  and of type $S_+$ on $C$. Then $G_T$ is of type $S_+,$ and $G_T(\cdot, y)$ is sequentially upper semicontinuous for every $y\in C.$ In particular, $G_T$ is Brezis pseudomonotone on $C \times C$.
%\end{proposition}

%It is time to study under what conditions in T defined on a real Banach space, $G_T$ inherits the properties c, d, and e in Theorem \ref{FirstTheoremBanach}.

\begin{proposition}\label{propvarineq}
	Let $D$ be a nonempty bounded subset of a separable normed space $X$ and   $T:X\rightrightarrows X^*$ with $D \subseteq  \rm{dom}(T)$.
	Suppose   that
	\begin{enumerate}[(i)]
		\item $T$ is s-w$^*$-closed on $D$;
		\item $T$ is bounded;
		\item $T$ is of type $S_+$ on $D$.
	\end{enumerate}
	Then $G_T$ is of type $S_+$ on $D$, $G_T(\cdot, y)$ is sequentially upper semicontinuous for every $y\in D$ and
	$$
	|G_T(z,y)-G_T(z,x)| \le h(z) ||y-x||, \quad \text{ for all } x,y,z \in D
	$$
	where $h: \rm{dom}(T) \to \R_+$ is bounded on bounded sets.
	%bounded sets.
	%In particular, $G_T$ is Brezis pseudomonotone on $\Phi(C) \times \Phi(C)$.
\end{proposition}
\begin{proof}
	The approach to proving this proposition is the same as the strategy in proof of Proposition 7 and the first part of Theorem 5 in \cite{Bikapi21}, but the assumptions are somewhat different and we prefer to repeat it here for the reader's convenience.

	First note that, for every $x \in D$, (i) implies that $T(x) $ is sequentially weak$^*$ closed. Moreover, by (ii)  $T(x)$ is bounded for all $x \in D$. Since $X$ is separable, $T(x)$ is weak$^*$ closed, hence it is weak$^*$  compact and also sequentially weak$^*$ compact.

	To prove that $G_T$ is of type $S_+$ on $D$, let us consider an  arbitrary sequence $\{x_n\} \subseteq D$, such that $x_n\to^{w} x\in D$, and $\liminf_{n\rightarrow \infty}G_T(x_n, x)\geq 0$. Then, by definition of $G_T$, we have
	$$\liminf_{n\rightarrow \infty} \sup_{x_n^*\in T(x_n)}\langle x_n^*, x-x_n\rangle \geq 0.$$
	
	Weak$^*$ compactness of $T(x_n)$ implies the existence of $z_n^*\in T(x_n)$ such that
	$$\sup_{x_n^*\in T(x_n)}\langle x_n^*, x-x_n\rangle= \langle z_n^*, x-x_n\rangle,$$
	and therefore, $\liminf_{n\rightarrow \infty}\langle z_n^*, x-x_n\rangle \geq 0$. As $T$ is of type $S_+$ on $D$, we deduce $x_n\rightarrow ^s x$ and this proves that $G_T$ is of type $S_+$ on $D$.
	
	To show now that $G_T(\cdot, y)$ is sequentially upper semicontinuous for every $y \in D$, consider an arbitrary sequence $\{x_n\}\subset D$  such that $x_n\to ^s x\in D$. By contradiction, suppose that  there exists $\overline{y}\in D$ such  that
	\begin{eqnarray}\label{propertyG_T}
	G_T(x, \overline{y})< \limsup_{n\rightarrow +\infty}G_T(x_n, \overline{y})=\lim_{k\rightarrow +\infty}G_T(x_{n_k}, \overline{y}),
	\end{eqnarray}
	where $\{x_{n_k}\}$ is a subsequence of $\{x_n\}$. Weak$^*$ compactness of $T(x_{n_k})$ implies that there exists $x_{n_k}^*\in T(x_{n_k})$ so that
	$$G_T(x_{n_k}, \overline{y}) = \langle x_{n_k}^*, \overline{y}-x_{n_k}\rangle.$$
	Since $T(\{x_n\})\subset T(D)$ is bounded and $X$ is separable ,  it follows that $x_{n_k}^*\to^{w^*} x^*$ (without loss of generality). Eventually by s-w$^*$- closedness of $T$, $x^*\in T(x)$ is deduced. Hence,
	$$\lim_{k\rightarrow +\infty}G_T(x_{n_k}, \overline{y}) = \lim_{k\rightarrow +\infty}\langle x_{n_k}^*, \overline{y}-x_{n_k}\rangle= \langle x^*, \overline{y}-x\rangle\leq G_T(x, \overline{y}),$$
	consequently,  \eqref{propertyG_T} is contradicted.
	%By Remark \ref{remarkUSCPseumonotone} ii., $G_T$ is B-pseudomonotone on $C\times C$ at last.
	
	Finally, given $z \in D$,
	%by similar approach in the first step of {\cite[Theorem 5]{Bikapi21}} and the assumptions on $T$ specified in the previous proposition, we have
	%$$
	%|G_T(z,y)-G_T(z,x)| \le h(z) ||y-x||,
	%$$
	%where $h:C \to \R$ is positive and bounded on bounded sets. %(see the proof of Theorem 5 in \cite{Bikapi21}).
	%We can review this method here:
	by the weak$^*$  compactness of $T(z)$, there exists $u^*\in T(z)$ such that $G_T(z, y)= \langle u^*, y-z\rangle$. Then
	
	\begin{eqnarray*}
		G_T(z, y) - G_T(z, x)&\leq& \langle u^*, y-z\rangle - \langle u^*, x-z\rangle = \langle u^*, y-x\rangle \\
		&\leq& \sup_{v^*\in T(z)}\|v^*\| \|y-x\|,
	\end{eqnarray*}
	and consequently, denoting by $h(z) = \sup_{v^*\in T(z)}\|v^*\|$ we have
	$$G_T(z, y) - G_T(z, x)\leq h(z) \|y-x\|.$$
	Since $T$ is bounded,   the function $h:  \rm{dom}(T) \rightarrow \R_+$ is  bounded on bounded sets.  Therefore, the claim is true.
\end{proof}

\begin{remark}\label{Rem not separable}
	
	A key point of the proof of the previous proposition is a property of the space $X^*$. Namely, $X^*$ should be such that a norm bounded and sequentially weak$^*$ closed set is also a sequentially weak$^*$ compact.

	Indeed, if this property holds, the proof of Proposition \ref{propvarineq} works well since $T(x)$ is sequentially weak$^*$ compact and a functional $y \in X$ attains its maximum on $T(x)$.
	
	In Proposition \ref{propvarineq}, separability of $X$ entails the validity of the property mentioned above. Nevertheless, if we suppose that $X$ is a Banach space, this property  can be obtained under different assumptions (not related to separability of $X$). Indeed, let us suppose that $X$ is such that $X^*$ does not contain an isomorphic copy of the Banach space $\ell_1$. Then Rosenthal's $\ell_1$ Theorem (see Chapter XI in \cite{Diestel84}) implies that each bounded sequence $\{x^*_n\}$ in $X^*$ has a weakly Cauchy subsequence $\{x^*_{n_k}\}$ (i.e. $\left\lbrace \langle x^{**},x^*_{n_k}\rangle\right\rbrace$ converges for every $x^{**} \in X^{**}$). Then $\{x^*_{n_k}\}$ is also a weakly$^*$ Cauchy sequence (i.e. $\left\lbrace \langle x^*_{n_k}, x \rangle\right\rbrace$ converges for every $x \in X$) and it is weakly$^*$ convergent since $X^*$ is weakly$^*$ sequentially complete (see Corollary 2.6.21 in \cite{Megginson}).
	
	In particular, in any reflexive space it happens that a norm bounded and sequentially weak$^*$ closed set is also a sequentially weak$^*$ compact. For other conditions that implies the sequential weak$^*$ compactness of bounded and weakly$^*$ closed set see Chapter XIII in \cite{Diestel84}.
	
\end{remark}

With support of the above-noted connection between the properties of $T$ and $G_T,$ we can now state the following results regarding the existence of a projected solution to $QVI(T,\Phi)$ in a normed space. The proof follows from Theorem \ref{FirstTheoremBanach} and  Proposition \ref{propvarineq} applied to  $D = \Phi(C)$.

\begin{theorem}\label{FirstTheoremBanachVarIneq}
	Let $C$ be a  nonempty, convex, and weakly compact subset of a separable normed space $X$. Let $\Phi: C \rightrightarrows X$
	%satisfisfy assumptions (i)-(iv) of Theorem \ref{FirstTheoremBanach}.
	be such that
	\begin{enumerate}[(i)]
		\item $\Phi(x)$ is nonempty and convex for every $x \in C$;
		\item  $\Phi(C)$ is relatively weakly compact;
		\item $\Phi$ is  sequentially weakly lower semicontinuous;
		\item $\Phi$ is weakly closed.
	\end{enumerate}
	Let $T:X\rightrightarrows X^*$ with  $\Phi(C)\subseteq \mathrm{dom}(T)$, satisfying the following properties:
	
	\begin{enumerate}[(a)]
		\item $T$ is s-w$^*$-closed on $\Phi(C)$;
		\item $T$ is bounded;
		\item $T$ is of type $S_+$ on $\Phi(C)$.
	\end{enumerate}
	If the sequence $\{x_n\}$ generated by Algorithm 1 is asymptotically regular, then it admits a weak limit point which is a projected solution of $QVI(T,\Phi)$.
\end{theorem}
%\begin{proof}
%Since using Proposition \ref{propvarineq}, the conditions on $C$, $\Phi$, and $T$ meet the assumptions of Theorem \ref{FirstTheoremBanach}, the proof is complete.
%\end{proof}

By applying Proposition \ref{propvarineq} with $D=\Phi(C)$ and Corollary \ref{SecondTheoremBanach}, we get easily the following.

\begin{corollary}\label{SecondTheoremBanachVarIneq}
	Let $C$ be a  nonempty, convex, bounded, and boundedly w-compact subset of a separable normed space $X$. Let $\Phi: C \rightrightarrows X$ be a set-valued map with nonempty and convex values while $\overline{\Phi(C)}$ is bounded and boundedly w-compact.
	Let $T:X \rightrightarrows X^*$ be a set-valued map with  $\Phi(C)\subseteq \mathrm{dom}(T)$. Suppose that the conditions (iii)-(iv) and (a)-(c)  of Theorem \ref{FirstTheoremBanachVarIneq} are satisfied. If the sequence $\{x_n\}$ generated by Algorithm 1 is asymptotically regular, then it admits a weak limit point which is a projected solution of $QVI(T,\Phi)$.
\end{corollary}

In view of Remark \ref{Rem not separable}, we have that Theorem  \ref{FirstTheoremBanachVarIneq} and Corollary \ref{SecondTheoremBanachVarIneq} hold true whenever a bounded and weak$^*$ closed set is also sequentially weak$^*$ compact. Therefore, in particular, we have the following result in reflexive spaces whose proof can be carried out with similar steps as the previous ones.

\begin{theorem}\label{ProjSolForT}
	Let $C$ be a nonempty, convex, bounded, and closed set of a reflexive  space $X$. Let $\Phi: C \rightrightarrows X$ be a set-valued map with nonempty, and convex values, satisfying conditions (iii)-(iv) of Theorem \ref{FirstTheoremBanach} and  with $\Phi(C)$ bounded.
	Let   $T:X\rightrightarrows X^*,$ such that $\Phi(C)\subseteq \mathrm{dom}(T)$.
	Suppose that $T$ is  s-w-closed, bounded  and of type $S_+$ on $\Phi(C)$. If the sequence $\{x_n\}$ generated by Algorithm 1 is asymptotically regular, then it admits a weak limit point which is a projected solution of $QVI(T,\Phi)$.
\end{theorem}

\begin{remark}
	If $T(x) = x$, then $G_T(x, y)= \langle x, y-x\rangle$. Therefore, the last two  examples of the previous section works also in this setting.
\end{remark}

\section{Conclusions}
Our study analyzes quasi equilibrium problems with non-self constraint map. We investigated this class of problems by using the concept of a projected solution and proved its existence in normed spaces. In fact, based on the features of the projection map, and utilizing an iterative algorithm, we were able to reach some findings about the existence of projected solutions for quasi equilibrium problems under the assumption of asymptotic regularity of the generated sequence.
%Our work confirms the important role of obtained results in study of existence of projected solutions in quasi variational inequalities.
%Compared to current studies in this context, we devoted particular attention to the domain of constraint maps by loosening the associated assumptions.
As part of our future goal, we plan first to study new and weaker sufficient conditions providing the asymptotic regularity of the sequence $\{x_n\}$ generated by Algorithm 1. Later, we would  modify our strategy and make use of fixed point theorems for the composition of the  projection map and the solution map while maintaining the same problem assumptions or, even better, trying to weaken them in order to study projected solutions in quasi equilibrium problems.

%\textbf{Data sharing is not applicable to this article as no datasets were generated or analysed during the current study.}


\begin{thebibliography}{plain}
	
	
	\bibitem{ManRia} Ait~Mansour~M, Riahi~H. Sensitivity analysis for abstract equilibrium problems. J. Math. Anal. Appl. 2005;306:684--691.
	
	\bibitem{Alber} Alber~YI. Metric and Generalized Projection Operators in Banach Spaces: Properties and Applications. In: Kartosator~AG editor. Theory and Applications of Nonlinear Operators of Accretive and Monotone Type. Dekker, New York. 1996. p. 15-50.
	
	\bibitem{AubFra}  Aubin~JP, Frankowska~H. Set-valued analysis. Birkhauser, Berlin; 1990.
	
	\bibitem{Milasi22} Allevi~E, De~Giuli~ME, Milasi~M, et~al. Quasi-variational problems with non self map on Banach spaces: Existence and applications. Nonlinear Anal. Real World Appl. 2022;67.
	
	
	\bibitem{AusCo2013}  Aussel~D, Cotrina~J. Stability of quasimonotone variational inequality under sign continuity. J. Opt. Theory Appl. 2013;158(3):653--667.
	
	\bibitem{AusCotIus2017} Aussel~D, Cotrina~J, Iusem~AM. An Existence Result for Quasi-Equilibrium Problems. J. Convex Anal. 2017;24:55--66.
	
	\bibitem{AusSulVet2016} Aussel~D, Sultana~A, Vetrivel~V. On the existence of projected solutions of quasivariational inequalities and generalized Nash equilibrium problems. J. Optim. Theory Appl. 2016;170:818--837.
	
	
	\bibitem{Bikapi21} Bianchi~M, Kassay~G, Pini~R. Brezis pseudomonotone bifunctions and quasi equilibrium problems via penalization. J. Global Optim. 2022;82:483-498.
	
	
	\bibitem{BluOet94} Blum~E, Oettli~W. From optimization and variational inequalities to equilibrium problems. Math. Student. 1994;63:123--145.
	
	
	\bibitem{BorZhu} Borwein~JM, Zhu~QJ. Techniques of variational analysis: CNS Books in Mathematics. Canadian Methamtical Society: Springer; 2005.
	
	
	\bibitem{BreNiSta72} Brezis~H, Nirenberg~L, Stampacchia~G. A remark on Ky Fan's minimax principle. Boll. Un. Mat. Ital. 1972;6:293--300.
	
	
	\bibitem{BuCo21} Bueno~O,  Cotrina~J. Existence of projected solutions for generalized Nash equilibrium problems. J. Optim. Theory Appl. 2021;191:344--362.
	
	
	\bibitem{ChaWoYao03}  Chadli~O, Wong~NC, Yao~JC. Equilibrium problems with applications to eigenvalue problems. J. Optim. Theory Appl. 2003;117:245--266.
	
	
	\bibitem{CharalAliprantis} Charalambos~D, Aliprantis~B. Infinite Dimensional Analysis: A Hitchhiker's Guide. 2nd ed. SpringerVerlag Berlin: Heidelberg GmbH $\&$ Company KG; 2006.
	
	\bibitem{CoZu2019} Cotrina~J, Z\'{u}\~{n}iga~J. Quasi-equilibrium problems with non-self constraint map. J. Global Optim. 2019;75:177--197.
	
	\bibitem{Cub97} Cubiotti~P. General Nonlinear Variational Inequalities with $S^1_+$ Operators. App. Math. Lett. 1997;10 (2):11--15.
	
	\bibitem{DanHad1999} Daniilidis~A, Hadjisavvas~N. Coercivity conditions and variational ineqaultie. Math. Program. 1999;86;443-428.
	
	\bibitem{Deu80} Deutsch~F. Existence of best approximation. J. Approx. Theory. 1980;28(2):132--154.
	
	\bibitem{Diestel84} Diestel~J. Sequences and Series in Banach Spaces: Graduate Texts in Mathematics. 92, New York: Springer; 1984.
	
	\bibitem{EfiSte1961}  Efimov~NV, Stechkin~SB.  Approximative compactness and Chebyshev sets.  Dokl. Akad. Nauk SSSR. 1961;140:522-524.
	
	
	\bibitem{Fan72} Fan~K. A minimax inequality and applications. Inequalities III Shisha (eds.). Academic Press; 1972. p. 103--113.
	
	
	\bibitem{Hu97} Hu~S, Papageorgiou~S. Handbook of Multivalued Analysis. Vol. I: Theory, Kluwer Academic Publisher; Norwell: MA; 1997.
	
	\bibitem{KanSteck19} Kanzow~Ch, Steck~D. Quasi-Variational Inequalities in Banach Spaces: Theory and Augmented Lagrangian Methods. SIAM J. Optim. 2019;29(4):3174--3200.
	
	\bibitem{Kon2019} Konnov~I. Equilibrium formulations of relatively optimization problems. Math. Methods Oper. Res. 2019;90:137--152.
	
	\bibitem{Len13} Lenzen~F, Becker~F, Lellmann~J, et~al. A class of quasivariational inequalities for adaptive image denoising and decomposition. Comp. Optim. Appl. 2013;54:371--398.
	
	\bibitem{Luc06} Lucchetti~R. Convexity and well-posed problems: CMS Books in Mathematics. Springer; 2006.
	
	\bibitem{Megginson} Megginson~RE. An Introduction to Banach Space Theory: Graduate Texts in Mathematics. New York: 183, Springer; 1998.
	
	\bibitem{Nad} Nadler~SBJR. Multi-valued contraction mappings. Pacific J. Math. 1969;30:475--488.
	
	\bibitem{Sion58} Sion~M. On general minimax theorems. Pac. J. Math. 1958:8;171--176.
	
	\bibitem{Tan1985} Tan~NX. Quasi-variational inequality in topological linear locally convex Hausdorff spaces. Math. Nachr. 1985;122(1):231--245.
	
	
	
\end{thebibliography}
\end{document}